\documentclass{IEEEojcsys}

\usepackage[colorlinks,urlcolor=blue,linkcolor=blue,citecolor=blue]{hyperref}

\usepackage[pdftex]{graphicx}
\usepackage{epstopdf}
\usepackage{amsmath}
\usepackage{amsfonts}
\usepackage{balance}

\usepackage{cite}

\usepackage{hyperref}
\usepackage{amssymb}
\usepackage{bbm}

\usepackage{array}
\usepackage{multirow}
\usepackage{color}
\usepackage{tablefootnote}
\usepackage{threeparttable}

\usepackage[mathscr]{euscript}

\usepackage{algorithm}
\usepackage{algpseudocode}

\jvol{XX}
\jnum{XX}
\paper{1234567}
\pubyear{20XX}
\receiveddate{05 June 2025}
\reviseddate{08 September 2025}
\currentdate{08 September 2025}
\doiinfo{OJCSYS.2021.Doi Number}

\newtheorem{assumption}{Assumption}
\newtheorem{corollary}{Corollary}
\newtheorem{definition}{Definition}
\newtheorem{proposition}{Proposition}
\newtheorem{remark}{Remark}
\newtheorem{theorem}{Theorem}

\DeclareMathOperator{\diag}{diag}

\DeclareMathOperator{\col}{col}
\DeclareUnicodeCharacter{2212}{-}
\newcommand{\jw}{\textcolor{black}}
\newcommand{\jdw}{\textcolor{black}}

\begin{document}

\sptitle{Article Category}

\title{Hybrid Data-enabled Predictive Control: Incorporating model knowledge into the DeePC} 


\author{J. D. WATSON\affilmark{1} (Member, IEEE)}



\affil{Department of Engineering, University of Canterbury, New Zealand. } 

\corresp{CORRESPONDING AUTHOR: J. D. Watson (e-mail: \href{mailto:jeremy.watson@canterbury.ac.nz}{jeremy.watson@canterbury.ac.nz})}

\markboth{HYBRID DATA-ENABLED PREDICTIVE CONTROL: INCORPORATING MODEL KNOWLEDGE INTO THE DEEPC}{J. D. WATSON}

\begin{abstract}
Predictive control can either be data-based (e.g. data-enabled predictive control, or DeePC) or model-based (model predictive control). In this paper we aim to bridge the gap between the two by investigating the case where only a partial model is available, i.e. incorporating model knowledge into DeePC. In our formulation, the partial knowledge takes the form of known state and output equations that are a subset of the complete model equations. We formulate an approach to take advantage of partial model knowledge which we call hybrid data-enabled predictive control (HDeePC). \jdw{We} prove feasible set equivalence and equivalent closed-loop behavior in the noiseless, LTI case. As we show, this has potential advantages over a purely data-based approach in terms of computational expense and robustness to noise in some cases. Furthermore, this allows applications to certain linear time-varying and nonlinear systems. Finally, \jdw{a number of case studies, including the control of an energy storage system in a microgrid, a triple-mass system, and a larger power system}, illustrate the potential of HDeePC.
\end{abstract}

\begin{IEEEkeywords}
Predictive control, data-enabled predictive control, energy systems
\end{IEEEkeywords}

\maketitle

\section*{NOMENCLATURE}
\addcontentsline{toc}{section}{NOMENCLATURE}
\jdw{\begin{IEEEdescription}[\IEEEusemathlabelsep\IEEEsetlabelwidth{$V_1,V_2,V_3$}]
\item[$\alpha$] piecewise nonlinear function in Case study 1b. 
\item[$A, B, C, D$] state-space matrices for the plant.
\item[$A_u, A_f$] decomposition of the matrix $A$ for HDeePC \eqref{eq:ss}. These matrices are unknown. 
\item[$A_c, A_\kappa$] decomposition of the matrix $A$ for HDeePC \eqref{eq:ss}. The matrices $A_c, A_\kappa$ are known.
\item[$A_y, C_y$] transformation matrices for HDeePC \eqref{eq:transform}.
\item[$C_u, C_f$] decomposition of the matrix $C$ for HDeePC \eqref{eq:ss}. These matrices are unknown. 
\item[$C_c, C_\kappa$] decomposition of the matrix $C$ for HDeePC \eqref{eq:ss}. The matrices $C_c, C_\kappa$ are known. 
\item[$\eta$] efficiency of the BESS. 
\item[$g$] decision variable in (H)DeePC. This is a vector of weights which (if a feasible $g$ can be found) ensures predicted trajectories are consistent with past data.
\item[$G$] plant $G(A,B,C,D)$.
\item[$h$] linear function for condensed HDeePC formulation \eqref{eq:hdpc_condensed}.
\item[$k$] time-step within the predictive control optimization problem.
\item[$\textbf{l}(G)$] lag of the plant $G$. 
\item[$L$] depth (number of block rows) of the Hankel matrix, $L = T_{ini} + N$.
\item[$\lambda_g, \lambda_y$] regularization weights.
\item[$m$] number of inputs to the plant.
\item[$n$] order (number of states) of the plant $G$. 
\item[$n_u$] number of states of the unknown part of the plant.
\item[$n_\kappa$] number of states of the unknown part of the plant. $n_u+n_\kappa = n$.
\item[$N$] horizon for predictive control.
\item[$p$] number of outputs from the plant.
\item[$p_u$] number of outputs from the unknown part of the plant.
\item[$p_\kappa$] number of outputs from the unknown part of the plant. $p_u+p_\kappa = p$.
\item[$Q$] output cost matrix.
\item[$r$] reference trajectory for the plant's outputs to follow.
\item[$R$] input cost matrix.
\item[$t$] time.
\item[$T$] data length in (H)DeePC.
\item[$T_{ini}$] number of past observations used in (H)DeePC.
\item[$\tau_q$] time constant of the BESS.
\item[$u$] input vector to the plant.
\item[$u^d$] collected input data.
\item[$u_{ini}$] past input trajectory.
\item[$u_i$] $i$th input channel.
\item[$U_P, U_F$] past and future inputs in DeePC (similar to $Y_P, Y_F$ for outputs).
\item[$x$] state vector of the plant.
\item[$x^*$] optimized variable $x$ (similarly for other variables).
\item[$\hat x$] measured (or estimated) state vector.
\item[$x_\kappa$] known state vector.
\item[$y$] output vector from the plant.
\item[$y^d$] collected output data.
\item[$y_{ini}$] past input trajectory (for DeePC).
\item[$y_{u,ini}$] past input trajectory (for HDeePC).
\item[$y_\kappa$] known output vector.
\item[$y_u$] unknown output vector.
\item[$Y_P, Y_F$] past and future outputs in DeePC (similar to $U_P, U_F$ for inputs).
\item[$Y_{U,P}, Y_{U,F}$] past and future outputs in HDeePC (replaces $Y_P, Y_F$).
\end{IEEEdescription}}

\section{INTRODUCTION}

Predictive control is one of the most important control techniques both in industry and academia. By solving an optimization problem to determine the control input to the plant, the cost can be minimized and constraints directly incorporated into the formulation. Traditionally, this has been based on the knowledge of the plant's state-space model, known as model predictive control (MPC). However, obtaining a model can be time-consuming and/or expensive \cite{hjalmarsson2005}. \jdw{System identification approaches (followed by MPC) are widely used and often effective; nevertheless there are cases where a data-driven form of predictive control \jdw{is} advantageous \cite{dorfler2023}.}

Data-enabled predictive control (DeePC) was proposed in 2019 by \cite{coulson2019, coulson2022}, using previously measured input/output trajectories to represent the plant, based on the Fundamental Lemma in behavioral systems theory \cite{markowsky2006, depersis2020}. This has \jdw{advantages} compared to MPC, as knowledge of the state-space model is no longer required. DeePC has been successfully applied to many real-world problems including power converters and power systems \cite{huang2019, huang2022}, traffic congestion control \cite{rimoldi2024}, quadcopters \cite{elokda2021}, and motor drives \cite{carlet2020}. \jdw{Although MPC is much more theoretically established, recent work has also demonstrated various stability and robustness guarantees for data-driven forms of predictive control which are attractive (e.g. \cite{berberich2021, berberich2025}).} 

However, there are weaknesses with DeePC, including: noise susceptibility, \jdw{difficulties applying DeePC to nonlinear systems (Willem's Fundamental Lemma applies to LTI systems only), computational cost \cite{li2025review, zhang2023, vahidimoghaddam2024onlinereducedorderdataenabledpredictive, zhou2025}, and sampling requirements. The weaknesses have been a focus of considerable recent work, e.g. \cite{berberich2021, huang2023} consider robustness to noise, \cite{zhang2023,vahidimoghaddam2024onlinereducedorderdataenabledpredictive} look at ways to improve computational performance, while others (e.g. \cite{lian2021,berberich2022,lazar2024,berberich2024,huang2024,naf2025}) have considered how data-driven control can be applied to nonlinear systems via use of nonlinear kernels, reformulations, or intelligent selection of the data used as predictor. This paper presents another technique to address some of these weaknesses, which is} via incorporating a known part of the model into data-driven predictive control.  While the theoretical results in this paper are relatively straightforward, the contributions are meaningful for various practical applications, including energy storage / power systems and certain classes of nonlinear or time-varying systems.

\jdw{Incorporating partial model knowledge into DeePC can be beneficial in several ways. First, it enables the exact representation of known dynamics, allowing the formulation to ``approach" MPC as more of the dynamics are known. This can improve performance and reduce computational expense, as is consistent with previous work \cite{morato2024} and is illustrated in our case studies. Furthermore, under certain conditions, nonlinear and time-varying dynamics can be incorporated, thereby avoiding the poor system representations that may arise when applying purely data-driven methods to such systems \cite{baros2022}. Additionally, model knowledge can help address practical sampling issues: when a system has both fast and slow dynamics, the sampling rate must be chosen to capture the fast dynamics, but this may render the slow dynamics difficult to observe in noisy, finite-precision measurements. Incorporating a model of either the slow or fast dynamics often alleviates this problem.}

\jdw{In this paper, we show that incorporating partial model knowledge is possible by posing both data-based and model-based representations as constraints to the predictive control optimization problem. We call the resulting algorithm hybrid data-enabled predictive control (HDeePC) and investigate the conditions under which this is possible. While incorporating partial model information into data-driven control has been successfully applied \cite{berberich2020, djeumou2022}, to the author's knowledge, the combination of model information with data-enabled predictive control in this fashion has only recently been explored in a few preprints; in this work \cite{watson2025hybriddataenabledpredictivecontrol} and later \cite{zieglmeier2025semidatadrivenmodelpredictivecontrol} and \cite{li2025mdrdeepcmodelinspireddistributionallyrobust}. \cite{zieglmeier2025semidatadrivenmodelpredictivecontrol} employs a different formulation based on an underlying parametric model to enhance robustness, while  MDR-DeePC \cite{li2025mdrdeepcmodelinspireddistributionallyrobust} focuses on empirical robustness via distributionally robust optimization. In contrast, this paper establishes the theoretical framework for incorporating partial model knowledge.} The paper's contributions may be summarized as follows: 
\begin{enumerate}
    \item We investigate the incorporation of partial model knowledge into DeePC, and propose HDeePC to achieve this. Our formulation generalizes predictive control between the two cases of DeePC (data-based representation) and MPC (discrete-time state-space equation representation). 
    \item We derive conditions under which we are able to prove feasible set equivalence to MPC/DeePC and equivalent closed-loop behavior for the proposed HDeePC.
    \item \jdw{We show that a restricted class of nonlinear systems can be incorporated into the HDeePC framework. }
\end{enumerate}

\jdw{We} validate our results and illustrate the \jw{practical advantages} of HDeePC with examples \jdw{(including two nonlinear cases) in Section \ref{cs}, for which code has been made publicly available. \jdw{We also conduct an empirical investigation into the number of known states and model mismatch issues.}}

The paper is organized as follows: in section \ref{problem} the problem is formulated along with a brief overview of both MPC and DeePC. We then present hybrid data-enabled predictive control in section \ref{hdpc}, showing in that in the noiseless LTI case, HDeePC results in feasible set equivalence and equivalent closed-loop behavior to either MPC or DeePC (which were shown to be equivalent to each other in \cite{coulson2019}). We then present \jdw{three case studies (a battery energy storage system (BESS), a triple mass system \cite{fiedler2021}, and an AC/DC power system) in Section \ref{cs}, with further discussion in Section \ref{discuss} and} concluding remarks provided in Section \ref{conc}.

\section{PROBLEM FORMULATION AND OVERVIEW}\label{problem}

\subsection{PROBLEM STATEMENT AND NOTATION}

We investigate the control of a linear discrete-time system $\jw{G}(A,B,C,D)$ of the form:
\begin{subequations}\label{eq:ss0}
    \begin{align}
        x(t+1) &= Ax(t) + Bu(t)\\
        y(t) &= Cx(t) + Du(t)
    \end{align}
\end{subequations}
with system matrices $A \in \mathbb{R}^{n\times n}$, $B \in \mathbb{R}^{n\times m}$, $C \in \mathbb{R}^{p\times n}$, and $D \in \mathbb{R}^{p\times m}$, and input $u(t) \in \mathbb{R}^m$, state $x(t) \in \mathbb{R}^n$ and output $x(t) \in \mathbb{R}^p$ at time $t \in \mathbb{Z}_{\geq 0}$. We denote the order of \eqref{eq:ss0} (assuming a minimal realization) by $\mathbf{n}(\jw{G}) \in \mathbb{Z}_{> 0}$ and its lag\footnote{See \cite[Section IV.B]{coulson2019} for a full definition of these concepts.} by $\mathbf{l}(\jw{G}) \in \mathbb{Z}_{> 0}$. As will be discussed later, we investigate the case where the matrices $A,B,C,D$ are \emph{partially known}. We will therefore consider systems with an unknown part and a known part. \jdw{We use $u$ and $\kappa$ subscripts to denote variables (and matrices, where possible) corresponding to the unknown and known part, respectively. To avoid confusion we therefore denote the value of a variable (e.g. $x$) at time-step $k$ by $x(k)$ rather than $x_k$ which may be confused for the known part of the states.}

Our objective is to design an appropriate input trajectory \jdw{$u(t)$} to \eqref{eq:ss0} in order to track some given reference trajectory \jdw{$r(t)$} while satisfying constraints on input \jdw{$u(t) \in \mathcal{U}$} and output \jdw{$y(t) \in \mathcal{Y}$} and minimizing an appropriate cost function. \jdw{We use $t$ for the time at which the predictive control problem is solved and $k$ for the time index within the horizon of the optimization problem.}


\subsection{BEHAVIORAL THEORY}

Since the fundamental definitions and concepts of behavioral systems theory and the fundamental lemma have appeared in many papers over the last few years, we only restate a few concepts we will use in the rest of the paper. An excellent overview of behavioral systems theory applied to predictive control may be found in \cite{coulson2019} which inspired the present work, and the following definitions and equations are largely restated from the same reference. 
\begin{definition}[Persistency of excitation] \cite[Definition 4.4]{coulson2019} \label{def:pe}
Let $L,T \in \mathbb{Z}_{>0}$ such that $T \geq L$. The signal $u = \textrm{col}(u(1), \dots,u(T)) \in \mathbb{R}^{Tm}$ is \textit{persistently exciting of order $L$} if the Hankel matrix
\begin{equation*}
    \mathcal{H}_L(u) :=\begin{pmatrix}u(1) & \dots & u({T-L+1})\\
    \vdots & \ddots &\vdots \\
    u(L) & \dots & u(T)\end{pmatrix}
\end{equation*}
is of full row rank. 
\end{definition}
\jdw{The term \textit{persistently exciting} refers to input data that is sufficiently rich and long (Definition \ref{def:pe}) such that when this input excites the system, an output sequence is yielded that fully represents the system. If so, the \textit{persistency of excitation} condition is satisfied. This allows a data-based representation of the system, as follows. Suppose} this condition is satisfied for some input sequence $u^d$ to system \eqref{eq:ss0} and input/output data $\col(u^d, y^d)$ is collected from applying $u^d$ to \eqref{eq:ss0} and recording the output $y^d$. \jdw{Furthermore,} $\col(u^d, y^d)$ is partitioned as follows \cite[(4)]{coulson2019} for some $N \in \mathbb{Z}_{> 0}$:
\begin{equation}\label{eq:uypf}
    \begin{pmatrix}
        U_P \\ U_F
    \end{pmatrix} :=\mathcal{H}_{T_{ini} + N}(u^d), 
    \begin{pmatrix}
        Y_P \\ Y_F
    \end{pmatrix} :=\mathcal{H}_{T_{ini} + N}(y^d), 
\end{equation}
\jw{where $U_P \in \mathbb{R}^{mT_{ini} \times L}$ is the first $T_{ini}$ block rows of $\mathcal{H}_L(u^d)$, $U_F \in \mathbb{R}^{mN \times L}$ is the last $N$ block rows of $\mathcal{H}_L(u^d)$, $Y_P \in \mathbb{R}^{pT_{ini} \times L}$ and $Y_F \in \mathbb{R}^{pN \times L}$ likewise of $\mathcal{H}_L(y^d)$.} \jdw{Then,} any sequence $\col(u_{ini},u,y_{ini},y)$ is a trajectory of \eqref{eq:ss0} if and only if there exists a $g \in \mathbb{R}^{T-L+1}$ such that \cite[(5)]{coulson2019}
\begin{equation}\label{eq:hdpcdata0}
         \begin{bmatrix} U_P  \\ Y_P \\ U_F \\ Y_F\end{bmatrix}g = \begin{bmatrix} u_{ini} \\ y_{ini} \\ u \\ y\end{bmatrix}.\\
\end{equation}
\jdw{where the vectors $u_{ini}$, $y_{ini}$ respectively denote the $T_{ini}$ most recent input and output measurements.}

\jw{As} (verbatim) in \cite{coulson2019}, we denote the lower triangular Toeplitz matrix consisting of $A, B, C, D$ as:
\begin{equation*}
    \mathscr{T}_N\jw{(A,B,C,D)} :=\begin{pmatrix} D & 0 &\dots & 0\\
    CB & D \dots & 0 \\
    \vdots & \ddots \ddots & \vdots \\
    CA^{N-2}B & \dots & CB & D\end{pmatrix}
\end{equation*}
and the observability matrix as:
\begin{equation*}
    \mathscr{O}_N(A,C) := \col(C,CA,\dots,CA^{N-1}).
\end{equation*}

\subsection{MODEL PREDICTIVE CONTROL}

If the system is known, model predictive control can be used. We (mostly) follow the notation in \cite[equation (2)]{coulson2019} in stating the optimization problem:
\begin{equation}\label{eq:mpc}
\begin{aligned}
\min_{u,x,y} \quad & \sum_{k=0}^{N-1}\Vert y(k)-r(k)\Vert^2_Q + \Vert u(k)\Vert^2_R\\
  &x(0) = \hat x(t) \\
  &x(k+1) = Ax(k) + Bu(k), \forall k \in \{0, \dots, N-1\}    \\
  &y(k) = Cx(k) + Du(k), \forall k \in \{0, \dots, N-1\}  \\
  &u(k) \in \mathcal{U{}}, \forall k \in \{0, \dots, N-1\}    \\
  &y(k) \in \mathcal{Y{}}, \forall k \in \{0, \dots, N-1\}    \\
\end{aligned}
\end{equation}

where $N \in {\mathbb{Z}_{> 0}}$ is the time horizon, $u = (u(0), u(1), \dots,$
$u(N-1))$, $x = (x(0), x(1), \dots,x(N-1))$, and $y = (y(0), y(1), \dots,y$$(N-1))$ are the decision variables, $r =  (r(t), r(t+1), \dots,r(t+N-1))$ is the desired reference trajectory throughout the control horizon, $\mathcal{U} \subseteq \mathbb{R}^m$ and $\mathcal{Y} \subseteq \mathbb{R}^p$ are input and output constraint sets respectively, and $Q \in \mathbb{R}^{p\times p}_{\geq 0}$ and $R \in \mathbb{R}^{m\times m}_{> 0}$ are the output and state cost matrices respectively. $\hat x(t)$ is the measured \jdw{(or estimated)} state vector at time $t$, the time at which the optimization problem is solved. As is clear from \eqref{eq:mpc}, the state matrices $(A, B, C, D)$ must be known.

\subsection{DATA-ENABLED PREDICTIVE CONTROL}

Replacing the model equations in \eqref{eq:mpc} with linear constraints based on the collected data gives the data-enabled predictive control problem \cite[equation (6)]{coulson2019}:
\begin{equation}\label{eq:deepc}
\begin{aligned}
\min_{g,u,y} \quad & \sum_{k=0}^{N-1} \Vert y(k) - r(t+k) \Vert_Q^2 + \Vert u(k)\Vert_R^2\\
\textrm{s.t.} \quad & \begin{bmatrix} U_P \\ Y_P \\ U_F \\ Y_F\end{bmatrix}g = \begin{bmatrix} u_{ini} \\ y_{ini} \\ u \\ y\end{bmatrix}\\
  &u(k) \in \mathcal{U}, \forall k \in \{0, \dots, N-1\}    \\
  &y(k) \in \mathcal{Y}, \forall k \in \{0, \dots, N-1\}    
\end{aligned}
\end{equation}
where $g \in \mathbb{R}^{T - T_{ini} - N + 1}$, $U_P, U_F, Y_P, Y_F$ are past/future input/output data as described in \eqref{eq:hdpcdata0}, and other notation is identical to \eqref{eq:mpc}.

\section{HYBRID DATA-ENABLED PREDICTIVE CONTROL}\label{hdpc}

\subsection{FORMULATION}
Let us start by considering a system \jdw{comprised of an unknown and known subsystem as follows}: 
\begin{subequations}\label{eq:ss}
    \begin{align}
        x(t+1) &= \begin{bmatrix}
            x_u(t+1) \\
            x_\kappa(t+1)
        \end{bmatrix} =  \begin{bmatrix}
            A_u & A_f \\
            A_c & A_\kappa
        \end{bmatrix}x(t) + \begin{bmatrix}
            B_u \\
            B_\kappa
        \end{bmatrix}u(t) \\
        y(t)  &= \begin{bmatrix}
            y_u(t) \\
            y_\kappa(t)
        \end{bmatrix} =  \begin{bmatrix}
            C_u & C_f \\
            C_c & C_\kappa
        \end{bmatrix}x(t) +  \begin{bmatrix}
            D_u \\
            D_\kappa
        \end{bmatrix}u(t) 
    \end{align}
\end{subequations}

where $A_u \in \mathbb{R}^{n_u \times n_u}$ etc., and \jdw{$n_\kappa$ and $n_u$ ($p_\kappa$ and $p_u$)} are the number of known and unknown states (outputs) respectively. \jdw{We assume that the system matrices 
$A_c, A_\kappa, B_\kappa, C_c, C_\kappa, D_\kappa$ 
are known a priori, and that the vectors $y_u$ and $y_\kappa$ 
correspond to disjoint subsets of the system outputs.\footnote{\jdw{This assumption avoids infeasibility arising from inconsistencies 
between model-based and data-driven predictions of the same output trajectory.}}} We start by applying a data-based representation for the unknown dynamics:
\begin{equation}\label{eq:hdpcdata}
         \begin{bmatrix} U_P  \\ Y_{U,P} \\ U_F \\ Y_{U,F}\end{bmatrix}g = \begin{bmatrix} u_{ini} \\ y_{u,ini} \\ u \\ y_u\end{bmatrix}\\
\end{equation}
where $Y_{U,P}$ and $Y_{U,F}$ are the partitioned Hankel matrix (analogously to \eqref{eq:uypf}) of the output data $y^d_u$ collected from the system \eqref{eq:ss0}; in this case only the outputs in $y_u$ are included. We also include a state-space system for the known dynamics:
\begin{equation}\label{eq:hdpcss_pre}
\resizebox{1\linewidth}{!}{%
$\displaystyle
\begin{aligned}
     x_\kappa(0) &= \hat x_\kappa(t) \\
     x_\kappa(k+1) &= A_c x_u(k) + A_\kappa x_\kappa(k) + B_\kappa u(k), \\&\quad\quad\quad\quad\quad\quad\quad\quad\quad\quad\quad \forall k \in \{0,\dotsc,N-1\} \\
     y_\kappa(k) &= C_c x_u(k) + C_\kappa x_\kappa(k) + D_\kappa u(k), \\&\quad\quad\quad\quad\quad\quad\quad\quad\quad\quad\quad \forall k \in \{0,\dotsc,N-1\}
\end{aligned}
$%
}\end{equation}

An obvious issue is that since there is no explicit representation in \eqref{eq:hdpcdata} for $x_u$, $x_\kappa$ and $y_\kappa$ cannot easily be a function of $x_u$ \jw{(i.e. if $A_c, C_c$ are not zero matrices, \eqref{eq:hdpcss_pre} requires knowledge ($x_u$) which is not available). In the case of coupled states/outputs with non-zero $A_c$ and/or $C_c$, our approach is to transform the system to use the information in $y_u$ instead (if possible - see Remark \ref{rmk:transform}), via the following equations:}
\jw{\begin{subequations}\label{eq:transform}
    \begin{align}
        A_c &= A_yC_u, \quad A_yC_f = 0_{n_\kappa\times n_\kappa}, \quad A_yD_u = 0_{n_\kappa\times m}\label{eq:transforma}\\
        C_c &= C_yC_u, \quad C_yC_f = 0_{p_\kappa\times n_\kappa}, \quad C_yD_u = 0_{p_\kappa\times m}.\label{eq:transformc}
    \end{align}
\end{subequations}}
We are now ready to state the assumption that is necessary for HDeePC to be applied to \eqref{eq:ss}.
\begin{assumption}\label{as:transform}
    \jw{(A) Given the system \eqref{eq:ss}, a matrix $A_y$ that satisfies \eqref{eq:transforma} exists and is known. (B) Likewise, a matrix $C_y$ that satisfies \eqref{eq:transformc} exists and is known.}
\end{assumption}

\jdw{\normalsize As above, we assume $A_y$ and $C_y$ are known a priori, but if not, equations \eqref{eq:transforma} and \eqref{eq:transformc} may be solved as feasibility problems with $A_y$ and $C_y$ as decision variables. However, if \eqref{eq:transform} has to be explicitly solved, knowledge of $C_u, C_f$ and $D_u$ (which are otherwise assumed unknown) is additionally required.}

If Assumption \ref{as:transform} holds, we can substitute \eqref{eq:transform} into \eqref{eq:hdpcss_pre}:
\begin{equation}\label{eq:hdpcss}
\resizebox{1\linewidth}{!}{%
$\displaystyle
\begin{aligned}
     x_\kappa(0) &= \hat x_\kappa(t) \\
     x_\kappa(k+1) &= A_y y_u(k) + A_\kappa x_\kappa(k) + B_\kappa u(k), \\&\quad\quad\quad\quad\quad\quad\quad\quad\quad\quad\quad \forall k \in \{0, \dotsc, N-1\} \\
     y_\kappa(k) &= C_y y_u(k) + C_\kappa x_\kappa(k) + D_\kappa u(k),
     \\&\quad\quad\quad\quad\quad\quad\quad\quad\quad\quad\quad \forall k \in \{0, \dotsc, N-1\}.
\end{aligned}
$}
\end{equation}

\begin{remark}\label{rmk:transform}
\jdw{\textit{Restrictiveness of Assumption \ref{as:transform}:} The representation of $A_cx_u$ and $C_cx_u$ in terms of $A_yy_u$ and $C_yy_u$, respectively, is possible only if \eqref{eq:transform} holds. This is without loss of generality only when $C_u$ has full column rank and $C_f, D_u$ are zero. If Assumption \ref{as:transform}(B) is violated, the corresponding outputs can instead be included in the data-driven part \eqref{eq:hdpc}, disregarding prior knowledge of their equations. Although Assumption \ref{as:transform} imposes a strong structural condition and may be restrictive for some systems, it is satisfied by a range of practically relevant models, including the energy storage systems used in our first case study.}
\end{remark}

\begin{figure}[!ht]
    \centering
    \includegraphics[width=0.5\textwidth]{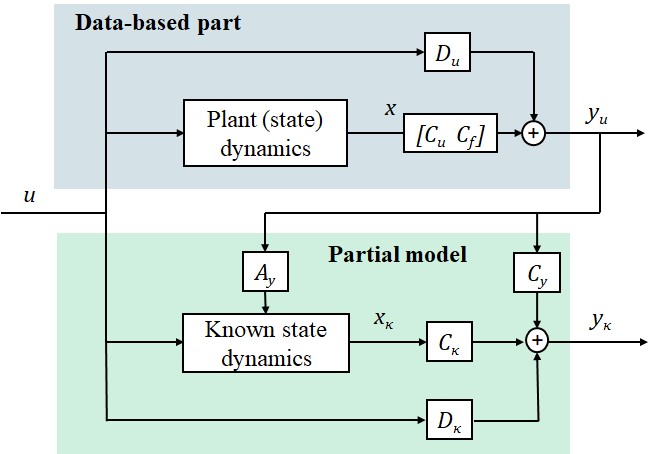}
    \caption{Illustration of formulation used in HDeePC.}
    \label{fig:hdpcsys}
\end{figure}

We illustrate the proposed formulation in Fig. \ref{fig:hdpcsys}. Given a time horizon $N > 0$, a reference trajectory $r = (r(t), r(t+1), \dots, r(t+N-1))$, and measurements $\hat x_\kappa(0)$ we formulate the following optimization problem for HDeePC:
\begin{equation}\label{eq:hdpc}
\begin{aligned}
\min_{g,u,x_\kappa,y_u, y_\kappa} \quad & \sum_{k=0}^{N-1} \Vert y(k) - r(t+k) \Vert_Q^2 + \Vert u(k)\Vert_R^2\\
\textrm{s.t.} \quad & \begin{bmatrix} U_P \\ Y_{U,P} \\ U_F \\ Y_{U,F}\end{bmatrix}g = \begin{bmatrix} u_{ini} \\ y_{u,ini} \\ u \\ y_u\end{bmatrix}\\
  &x_\kappa(0) = \hat x_\kappa(t) \\
  &x_\kappa(k+1) = A_yy_u(k) + A_\kappa x_\kappa(k) + B_\kappa u(k), \\
  & \quad \quad \quad \quad \forall k \in \{0, \dots, N-1\}    \\
  &y_\kappa(k) = C_yy_u(k) + C_\kappa x_\kappa(k) + D_\kappa u(k), \\
  & \quad \quad \quad \quad \forall k \in \{0, \dots, N-1\}  \\
  &u(k) \in \mathcal{U{}}, \forall k \in \{0, \dots, N-1\}    \\
  &y(k) \in \mathcal{Y{}}, \forall k \in \{0, \dots, N-1\}    \\
\end{aligned}
\end{equation}
where $y(k) = [y_u(k) \quad y_\kappa(k)]^T$.

If the systems were fully decoupled in their dynamics except that the same input is applied to both,  \eqref{eq:hdpc} finds an input that satisfies a separate DeePC and MPC problem with a blended cost function. However, the interesting part is when the known and unknown dynamics are coupled, in this case by matrices $A_f$ and $A_c$ as in \eqref{eq:ss} \jw{and Assumption \ref{as:transform} is required to cope with the coupling. Note that} we collect data from all inputs to the system for the data-based representation, but only the unknown outputs \eqref{eq:hdpcdata}. This ensures there are no inputs to the data-based part of the system which are not taken into account. The HDeePC problem \eqref{eq:hdpc} is then solved with a receding horizon (Algorithm \ref{alg:cap}), similarly to DeePC or MPC \cite{coulson2019,baros2022}.
\begin{algorithm}
\caption{HDeePC}\label{alg:cap}
\begin{algorithmic}
\Require Partial model ($A_c,A_\kappa,B_\kappa,C_c,C_\kappa,D_\kappa$), off-line data $\col(u^d, y^d_u)$ and most recent collected data $\col(u_{ini}, y_{u,ini})$, reference trajectory $r \in \mathbb{R}^{Np}$, input/output constraints $\mathcal{U},\mathcal{Y}$, cost function matrices $Q \geq 0, R > 0$.
\begin{enumerate}
    \item Solve \eqref{eq:hdpc} for $g^*$, $u^*$, $x_\kappa^*$, $y^*$.
    \item Apply input $(u^*(0),\dots,u^*(s))$ for some $s \leq N-1$ to the system. 
    \item Update $u_{ini}$ and $y_{ini}$ to the $T_{ini}$ more recent input/output measurements.
    \item Return to 1.
\end{enumerate}
\end{algorithmic}
\end{algorithm}

\begin{remark}\label{rmk:obs}
\jdw{\textit{Measurement requirements and observers:} In order for $C_u$ to be full column rank as mentioned in Remark~\ref{rmk:transform}, the unknown states must be directly measurable. Since the known states also appear in \eqref{eq:hdpc}, this condition typically amounts to requiring full state measurements, except for any states in $x_u$ that are completely decoupled from the known dynamics. An alternative to this stringent requirement is to employ a (partial) state observer\footnote{\jdw{E.g. $\hat x_\kappa(t+1) = A_\kappa\hat{x}_\kappa(t) + A_yy_u(t) + B_\kappa u(t) + L(y_\kappa(t) - (C_yy_u(t) + C_\kappa\hat{x}_\kappa(t) + D_\kappa u(t)))$ with appropriately designed observer gains $L$.}}. Estimated known states can then replace the true states in the model-based equalities in \eqref{eq:hdpc}, with the estimation error treated as an additive disturbance. Under standard detectability conditions this error is uniformly bounded, enabling the hybrid DeePC framework to be applied in practice with certainty-equivalent or robustified formulations. Obviously, unknown states cannot be estimated in this way, since an observer requires knowledge of their state equations; thus, this approach does not relax Assumption \ref{as:transform}.}
\end{remark}

\subsection{RELATION TO DEEPC AND MPC}

We now state two fairly obvious results for completeness in Propositions \ref{prop:deepc} and \ref{prop:mpc}. The propositions show that HDeePC is equivalent to either DeePC or MPC in the case that no model knowledge is available or full model knowledge is available, respectively. 

\begin{proposition} \label{prop:deepc}
Suppose there are no known states $x_\kappa$ and outputs $y_\kappa$ (for which the equation is known) and thus $x_\kappa$ and $y_\kappa$ are empty vectors. Then  \eqref{eq:hdpc} is equivalent to \eqref{eq:deepc}.
\end{proposition}
\begin{proof}
Since $x_\kappa, y_\kappa$ are empty, $x = x_u$, $y = y_u$. The proposition follows trivially by removing all $x_\kappa$ and $y_\kappa$ terms from \eqref{eq:hdpc} and then substituting $x = x_u$, $y = y_u$.
\end{proof}

\begin{proposition} \label{prop:mpc}
Suppose there are no unknown states $x_u$ and outputs $y_u$, and thus $x_u$ is an empty vector and no data is collected (i.e. \eqref{eq:hdpcdata} is removed from the constraints of \eqref{eq:hdpc}). Then \eqref{eq:hdpc} is equivalent to \eqref{eq:mpc}.
\end{proposition}
\begin{proof}
Given the proposition, the data-based system representation (constraints) in \eqref{eq:hdpc} are removed. Since there cannot be outputs based on unknown states, $y_\kappa = y$ and $y_u$ is empty. Removing all $y_u$ terms and substituting $x = x_\kappa$, $A = A_\kappa, B = B_\kappa, C = C_\kappa, D=D_\kappa$ (since there are no unknown dynamics) in \eqref{eq:hdpc} completes the proof.
\end{proof}

\begin{remark}
\jdw{\textit{Relationship between DeePC, HDeePC, and MPC:}} Propositions \ref{prop:deepc} and \ref{prop:mpc} show that the HDeePC problem \eqref{eq:hdpc} generalizes predictive control between the two extreme cases of DeePC (no explicit model knowledge used) and MPC (no data used). 
\end{remark}

\subsection{EQUIVALENCE PROOF FOR THE LINEAR, NOISELESS CASE}
Under certain assumptions, we prove that the optimal control sequence $u^*$ which is solution to \eqref{eq:hdpc} is equivalent to that of \eqref{eq:mpc} (and therefore also of \eqref{eq:deepc}, as proven in \cite{coulson2019}). This is an intuitive result given \cite[Theorem 5.1 and Corollary 5.1]{coulson2019}, but is given for completeness. \jw{We state additional assumptions first relating to the persistency of excitation of the collected data and the partial model knowledge.}

\begin{assumption}\label{as:pe}
\jw{We assume that the (noise-free) data collected from system \eqref{eq:ss} $\col(u^d, y_u^d)$ in $\col(U_P,Y_{U,P},U_F,Y_{U,F})$ is such that $u^d$ is persistently exciting of order $T_{ini} + N + \mathbf{n}(\jw{G})$ with $T_{ini} \geq \mathbf{l}(\jw{G})$.}
\end{assumption}

\begin{assumption}\label{as:knowntx}
\jw{We assume that $\hat x = x_{ini}$, and that $A_c, A_\kappa,A_y,$ $B_\kappa,C_\kappa,C_y,D_\kappa$ in \eqref{eq:ss} are known.}
\end{assumption}

\begin{theorem}[Feasible Set Equivalence]\label{thm:fse}
Consider the hybrid data-enabled predictive control optimization problem \eqref{eq:hdpc} in conjunction with a controllable LTI system $\jw{G}$ of the form \eqref{eq:ss} \jw{which satisfies Assumption \ref{as:transform}. Assume that Assumptions \ref{as:pe} and \ref{as:knowntx} hold.} Then the feasible set of \eqref{eq:hdpc} is equivalent to that of \eqref{eq:mpc} and \eqref{eq:deepc}.
\end{theorem}
\begin{proof}
Please see the Appendix.
\end{proof}

\begin{corollary}[Equivalent Closed-Loop Behavior] \label{cor:eclb}
Consider Algorithm 1 with $Q \geq 0$, $R > 0$, and $\mathcal{U},{Y}$ convex, non-empty. Under the assumptions in Theorem \ref{thm:fse}, Algorithm 1 results in equivalent closed-loop behavior to \cite[Algorithm 1]{coulson2019}.
\end{corollary}

\begin{proof}
Given feasible set equivalence (Theorem \ref{thm:fse}), the proof is equivalent to that of \cite[Corollary 1]{coulson2019} and is therefore omitted. 
\end{proof}

\begin{remark}
\jw{\jdw{\textit{Recursive feasibility:}} As our formulation \eqref{eq:hdpc} does not include terminal ingredients, it is not possible to prove recursive feasibility in general (similarly to the original DeePC formulation in \cite{coulson2019}). However, in the noiseless case, since HDeePC, DeePC and MPC all have equivalent closed-loop behavior (Corollary \ref{cor:eclb}), adding e.g. an appropriate terminal constraint would allow recursive feasibility to be proven by applying conventional methods from MPC with terminal constraints, similarly to \cite[Section III.B]{berberich2021}) and subject to typical assumptions. Proving recursive feasibility in the noisy case is left to future work.}
\end{remark}

\begin{remark}
    \jdw{\textit{Measurement noise and model mismatch:} While Theorem 1 and Corollary 1 establish equivalence under ideal assumptions, in practice both measurement noise and model mismatch may affect the Hankel representation and the parametric subsystem. Their systematic treatment, possibly via robust MPC (e.g. \cite{kuntz2024}) or robust data-driven predictive control (e.g., \cite{berberich2021, huang2023}) techniques, remains an important direction for future work. In particular, rigorous guarantees that simultaneously address DeePC robustness with respect to noisy data in the Hankel matrices, and MPC robustness to plant-model mismatch are still lacking. Possible practical approaches include, for example, distributionally robust optimization \cite{li2025mdrdeepcmodelinspireddistributionallyrobust}.}
\end{remark}

\subsection{\jw{COMPUTATIONAL CONSIDERATIONS}}\label{sect:comp}

Comparing \eqref{eq:deepc} and \eqref{eq:hdpc}, the dimension of the optimization problem is changed (usually reduced in HDeePC, e.g. a reduction of 50 equality constraints in the first case study (Section \ref{sect:bess})). 
However, as with both DeePC and MPC, it is possible to condense HDeePC to a single decision variable $g$ via the equality constraints. We present a comparison of computational complexity (number of decision variables and constraints) between HDeePC and DeePC in this case. The DeePC formulation \eqref{eq:deepc} may be condensed and simplified (without bounds on $u$ and $y$, which would equivalent for both cases) as:
\begin{equation}\label{eq:condenseddpc}
\begin{aligned}
\min_{g} \quad & \Vert Y_Fg - r \Vert_\mathcal{Q}^2 + \Vert U_Fg\Vert_\mathcal{R}^2\\
\textrm{s.t.} \quad & \begin{bmatrix} U_P \\ Y_{P}\end{bmatrix}g = \begin{bmatrix} u_{ini} \\ y_{ini}\end{bmatrix}
\end{aligned}
\end{equation}

where $\mathcal{Q} \in \mathbb{R}^{pN\times pN}$ and $\mathcal{R} \in \mathbb{R}^{mN\times mN}$ are weighting matrices. To formulate the corresponding condensed HDeePC formulation, we can condense the model part via $u$ and therefore $g$, which we denote as $y_\kappa = h(x_{ini},g)$\footnote{The function $h (x_{ini}, g)$ is a linear function of the measured initial state $x_{ini}$ and $g$ as follows: $h (x_{ini}, g) = \mathscr{O}_N(A, \begin{bmatrix} C_c & C_\kappa \end{bmatrix})x_{ini} + \mathscr{T}_N(A, B, \begin{bmatrix} C_c & C_\kappa\end{bmatrix}, D_\kappa)U_Fg$.}:

\begin{equation}\label{eq:hdpc_condensed}
\begin{aligned}
\min_{g} \quad \Vert \begin{bmatrix} Y_{U,F}g \\ h (x_{ini}, g)\end{bmatrix}  &- r \Vert_\mathcal{Q}^2 + \Vert U_Fg\Vert_\mathcal{Q}^2\\
\textrm{s.t.} \quad \begin{bmatrix} U_P \\ Y_{U,P}\end{bmatrix}g &= \begin{bmatrix} u_{ini} \\ y_{u,ini}\end{bmatrix}
\end{aligned}
\end{equation}

While the decision variables are now reduced to the vector $g$ in both cases, DeePC has additional equality constraints from the larger size of $Y_P$ compared to $Y_{U,P}$. The precise dimensions (decision variables and constraints) of the problem are: 
\begin{itemize}
    \item DeePC: $T-T_{ini}-N+1$ decision variables, $(m+p)T_{ini}$ equality constraints.
    \item HDeePC: $T-T_{ini}-N+1$ decision variables, $(m+p_u)T_{ini}$ equality constraints.
\end{itemize}
Since $p_u < p$, the HDeePC problem has lower dimension, \jdw{and adding additional model knowledge (i.e. increasing $p_\kappa$ which decreases $p_u$ by the same amount) results in lower problem dimension by exactly $T_{ini}$ for each additionally known output equation.} Note however the linear equations in $h (\hat x_{ini}, g)$ do need to be computed for HDeePC and this computational expense will depend on the properties of the known dynamics. Our empirical results in Section \ref{sect:bess} suggest that HDeePC still generally preserves a computational advantage. 


\subsection{NONLINEAR SYSTEMS}

\jdw{In this section we show that HDeePC can be extended to a restricted class of nonlinear systems. }\jw{Consider the system \eqref{eq:nlss}, which may be viewed as a nonlinear extension of \eqref{eq:ss}:}
\jw{\begin{subequations}\label{eq:nlss}
\begin{align}
    x_u(t+1) &= A_ux_u(t) + A_fx_\kappa(t) + B_uu(t) \label{eq:nlssa}\\
    x_\kappa(t+1) &= f_\kappa(x_\kappa(t), y_u(t), u(t)) \\
    y_u(t) &= C_ux_u(t) + D_uu(t) \label{eq:nlssc}\\
    y_\kappa(t) &= g_\kappa(x_\kappa(t), y_u(t), u(t))
\end{align}
\end{subequations}}
where $f_\kappa$, $g_\kappa$ denote possibly non-linear functions. \jdw{We make an assumption which equates to the unknown dynamics being LTI in order to apply HDeePC.}

\begin{assumption} \label{as:nlequiv}
\jw{The system \eqref{eq:nlss} is linear time-invariant from input $u$ to output $y_u$.}
\end{assumption}

\begin{remark}
    \jw{\jdw{\textit{Restrictiveness of Assumption \ref{as:nlequiv}:}}  Assumption \ref{as:nlequiv} can be satisfied, e.g. if $A_f = 0$ or if $f_\kappa$ is linear. The idea is (if possible) to take the nonlinearity out of the data-driven part completely and represent the nonlinearity only in the known part. Assumption \ref{as:nlequiv} prevents couplings between unknown states and known nonlinear states, in order for the Fundamental Lemma to hold for the behavioral part of the system. This is a restricted class of nonlinear systems, but not without practical relevance, e.g. as seen in Case study 1b.}
\end{remark}

\jw{The MPC optimization problem for this case is:}
\jw{\begin{equation}\label{eq:mpcnl}
\begin{aligned}
\min_{g,u,x_\kappa,y_u, y_\kappa} \quad & \sum_{k=0}^{N-1} \Vert y(k) - r(t+k) \Vert_Q^2 + \Vert u(k)\Vert_R^2\\
\textrm{s.t.} \quad
  x(0) &= \hat x(t) \\
  \eqref{eq:nlss} &\text{ holds}, \forall k \in \{0, ..., N-1\}  \\
  u(k) &\in \mathcal{U{}}, \forall k \in \{0, ..., N-1\}    \\
  y(k)&\in \mathcal{Y{}}, \forall k \in \{0, ..., N-1\}    \\
\end{aligned}
\end{equation}}

The HDeePC optimization problem\footnote{\jw{It should be noted that both \eqref{eq:mpcnl} and \eqref{eq:hdpcnl} \jdw{are} now non-convex problems due to the addition of non-linear equality constraints.}} for this case is:
\begin{equation}\label{eq:hdpcnl}
\begin{aligned}
\min_{g,u,x_\kappa,y_u, y_\kappa} \quad & \sum_{k=0}^{N-1} \Vert y(k) - r(t+k) \Vert_Q^2 + \Vert u(k)\Vert_R^2\\
\textrm{s.t.} \quad &\begin{bmatrix} U_P \\ Y_{U,P} \\ U_F \\ Y_{U,F}\end{bmatrix}g = \begin{bmatrix} u_{ini}  \\ y_{u,ini} \\ u \\ y_u\end{bmatrix}\\
  x_\kappa(0) &= \hat x_\kappa(t) \\
  x_\kappa(k+1) &= f_\kappa(x_\kappa(k), y_u(k), u(k)), \forall k \in \{0, ..., N-1\}    \\
  y_\kappa(k) &= g_\kappa(x_\kappa(k), y_u(k), u(k)), \forall k \in \{0, ..., N-1\}  \\
  u(k) &\in \mathcal{U{}}, \forall k \in \{0, ..., N-1\}    \\
  y(k)&\in \mathcal{Y{}}, \forall k \in \{0, ..., N-1\}    \\
\end{aligned}
\end{equation}

\jdw{We are now ready to present the main result for nonlinear HDeePC in Proposition \ref{thm:nlequiv}.}

\begin{proposition} \label{thm:nlequiv}
Assume noise-free, persistently exciting (of a sufficient order) data $\col(u^d, y_u^d)$ in $\col(U_P,Y_{U,P},U_F,Y_{U,F})$ is collected from \eqref{eq:nlss}. Then, the feasible sets of  \eqref{eq:mpcnl} and \eqref{eq:hdpcnl} are equivalent.
\end{proposition}

\begin{proof}
\jw{Proposition \ref{thm:nlequiv} follows straightforwardly from the equivalence of the data-based representation \eqref{eq:hdpcdata} to \eqref{eq:nlssa}, \eqref{eq:nlssc} in the noiseless case.}
\end{proof}

\begin{remark}
    \jdw{\textit{General nonlinear HDeePC:} Despite the restrictiveness of the theoretical conditions (Assumption \ref{as:nlequiv}), nonlinear HDeePC \eqref{eq:hdpcnl} can in practice be applied to nonlinear systems which do not satisfy Assumption \ref{as:nlequiv}. Although theoretical guarantees are not available in this case, empirical results (as illustrated in Case studies 1b and 1c for example) are promising.}
\end{remark}

\section{CASE STUDIES}\label{cs}

The code for all case studies is available in the accompanying repository\footnote{https://github.com/jerrydonaldwatson/HDeePC}. All examples are solved by CVX (SDPT3; \jdw{Mosek for the larger example in Case study 3 due to better scaling and numerical stability}) in MATLAB 2023b on \jdw{an} Intel(R) Core(TM) i7-10750 @ 2.60GHz with 16GB RAM for \jdw{all methods. Note} that regularization and a slack variable $\sigma_y$ are added to both DeePC and HDeePC formulations. A variety of regularization terms were tried (an example of a different regularization is given on GitHub) and in these examples this choice did not significantly affect the computational expense for both DeePC and HDeePC. \jdw{Our system ID + MPC implementation uses subspace identification (n4sid) with the true state dimension given.}

\subsection{CASE STUDY 1a: BESS CONTROL IN DC MICROGRID (LINEAR CASE)}\label{sect:bess}

We consider a single-node (capacitance = 1 mF) DC microgrid connected to the main grid (constant voltage = 400 V) via a resistive-inductive (R = $2\Omega$, L = 5 mH) line. When the system is discretized (time-step of 1 ms), the following discrete-time system is obtained where $u_2$ is the known random ($\mathcal{N}(0,1)$ Amps, smoothed with a moving average filter of 10 samples) generation/load fluctuation (disturbance):
\begin{equation}\label{eq:ex1ss}
\begin{split}
    [A|B]     &= \left[\begin{array}{ccc|cc}
0.98 & 1 & 0 &  1 & 1 \\
-0.2 & 0.6 & 0 & 0 & 0 \\
0 &0 & 1 & -10^{-3}\tau_q^{-1} & 0 \\
    \end{array}
    \right] \\
    C &= \begin{bmatrix}
        1 & 0 & 0 \\
        0 & 0 & 1
    \end{bmatrix}, \quad D = 0_{2\times 2}.
\end{split}
\end{equation}

The output is subject to $\mathcal{N}(0,10^{-6})$ measurement noise. \jdw{This example primarily illustrates the difficulty of applying DeePC to a system with widely differing time-scales, although there is a modest computational benefit for HDeePC compared to DeePC as well. The example includes fast network voltage/current dynamics ($x_1, x_2$) and slow BESS state-of-charge $x_3$ dynamics. Although the difference in time-scales is significant,} DeePC ($N = 10, T = 200, T_{ini} = 50, Q = \diag(10^{-3},5\times 10^4), R = 10^{-3}, \lambda_g = 1, \lambda_y = 10^6)$ $\ell_1$-norm regularization is able to control this system satisfactorily when $\tau_q = 10^3$, reducing the state-of-charge toward the desired reference (0.5) while satisfying constraints (max BESS current is 5A and max node voltage deviation is $\pm 20$V) and minimizing costs. However, increasing the BESS capacity $\tau_q$ from $10^3$ to $10^4$ results in failure to track the desired state-of-charge (Fig. \ref{fig:bess3}) for DeePC. In contrast, our proposed method, HDeePC, is able to cope with this and \jdw{correctly keeps the BESS output current high (close to 5A, although limited by both current and voltage constraints) to discharge the battery as desired. HDeePC} is also computationally faster as predicted by Section \ref{sect:comp}. DeePC took \jw{an average of 1.141 seconds to solve each problem} while HDeePC took 1.003 seconds\footnote{The reported solve times are broadly consistent with those observed for SDPT3 on problems of similar size \cite{Toh2012}, especially when accounting for the overhead introduced by CVX’s high-level modeling and parsing. While more efficient implementations (e.g., direct solver calls or vectorized formulations) would likely reduce computation time, the results remain representative for prototyping and illustrative purposes within this framework.}, an 11.9\% reduction in computational time. \jdw{HDeePC was also compared to subspace ID + MPC, which took 1.029 seconds with a similar average cost\footnote{\jdw{The average cost is $2.1919\times 10^3$ for DeePC, $2.1895\times 10^3$ for HDeePC, and $2.1894 \times 10^3$ for subspace ID + MPC. Most of the cost is incurred by the state-of-charge error which cannot be reduced quickly}.}. } 
\begin{figure}[!ht]
    \centering
    \includegraphics[width=0.5\textwidth]{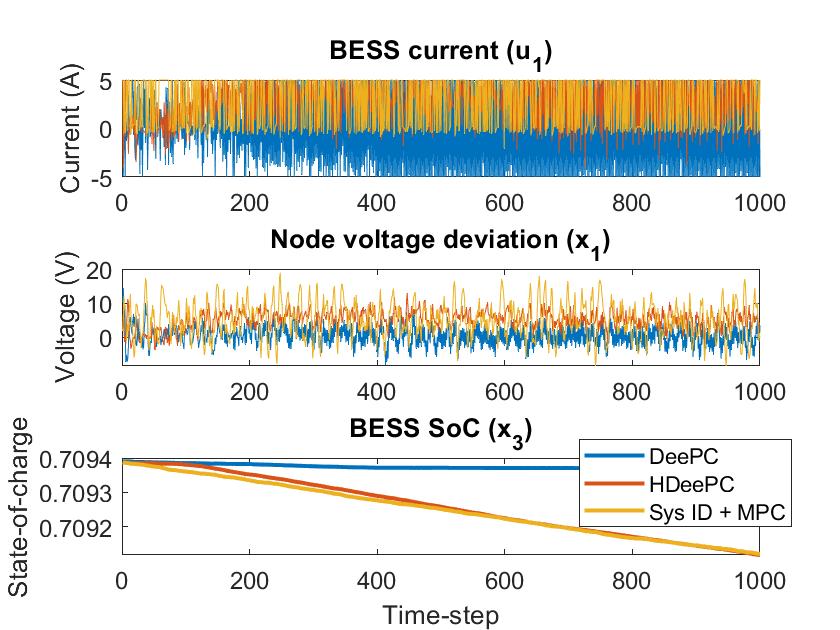}
    \caption{Case study 1a: BESS control using DeePC and HDeePC with a BESS SoC time constant of \jdw{$\tau_q=10^4 s^{-1}$.}}
    \label{fig:bess3}
\end{figure}

With quadratic rather than $\ell_1$-norm regularization \cite{MATTSSON2023625}, HDeePC again may be advantageous over DeePC: with $Q(2,2)$ increased to $5\times 10^6$, DeePC fails in reducing the SoC towards its reference value (0.5) when $\tau_q$ is increased to $10^6$, while HDeePC succeeds. A small computational advantage for HDeePC over DeePC also exists (a 4.0\% reduction in computational time in our experiments).

\subsection{CASE STUDY 1b: BESS CONTROL IN DC MICROGRID (NONLINEAR)}\label{sect:bessnl}

\jw{We now replace the third state equation in \eqref{eq:ex1ss} ($x_3(t+1) = x_3(t) - 10^{-3}\tau_q^{-1}u_1(t)$) with a (more realistic) nonlinear version, i.e. $x_3(t+1) = x_3(t) - 10^{-3}\alpha(u_1(t))\tau_q^{-1}u_1(t)$):}
\jw{\begin{equation}
    \alpha(u_1) = 
    \begin{cases}
        \frac{1}{\eta} & u_1 >= 0 \\
        \eta & \text{otherwise}
    \end{cases}
\end{equation}}
\jw{where $\eta \in [0, 1]$ is the efficiency of the BESS. We reduce $\tau_q$ to 10 in order to avoid the numerical issues in Case study 1a and both DeePC and HDeePC perform comparably for the linear case of $\eta = 1$. However, with $\eta = 0.9$, the system is now nonlinear, and DeePC fails to track the SoC reference while HDeePC \eqref{eq:hdpcnl}\footnote{\jw{For computational reasons, this is solved using an iterative convex approach. See code on GitHub: https://github.com/jerrydonaldwatson/HDeePC}} succeeds, as illustrated in Fig. \ref{fig:bessnl}. The average cost is $5.27$ for DeePC and $3.81$ for nonlinear HDeePC. \jdw{HDeePC's performance is also better than subspace ID + MPC (an average cost of $4.56$).}}
\begin{figure}[!ht]
    \centering
    \includegraphics[width=0.5\textwidth]{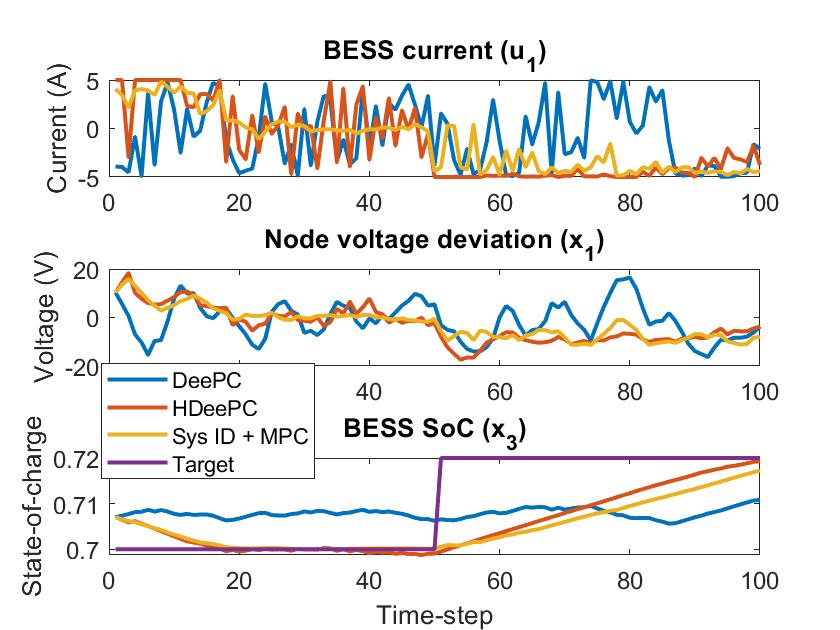}
    \caption{\jw{Case study 1b: BESS control using DeePC and HDeePC with a nonlinear model.}}
    \label{fig:bessnl}
\end{figure}

\subsection{\jdw{CASE STUDY 1C: UNKNOWN NONLINEAR DYNAMICS}}\label{sect:nl}
\jdw{We further modify the system to be more strongly nonlinear, including the unknown state equations ($x_1, x_2$ as before). In particular, the state equations are:
\begin{align*}
    x_1(t+1) &= a\sin(x_1(t)) + bx_1(t)u_1(t) + u_1(t) + u_2(t) \\
    x_2(t+1) &= a\sin(x_2(t)) + bx_2(t)u_2(t) + u_2(t) \\
    x_3(t+1) &= x_3(t) - 10^{-3}\alpha(u_1(t))\tau_q^{-1}u_1(t) \\
    y(t) &= Cx(t) 
\end{align*}
where $a = 0.9$, $b = 0.2$, and $\alpha(u_1)$, the matrix $C$, and $\tau_q$ are as defined previously.}

\begin{figure}[!ht]
    \centering
    \includegraphics[width=0.5\textwidth]{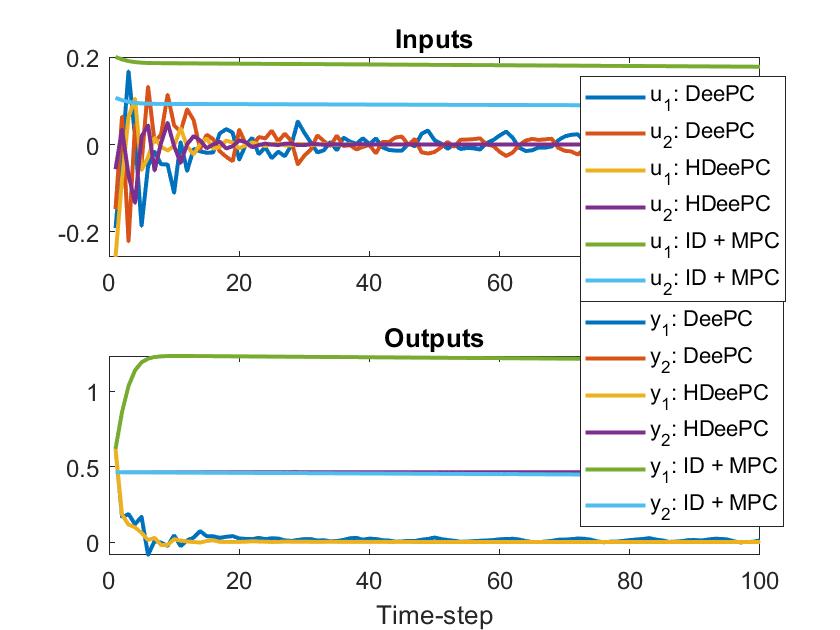}
    \caption{\jw{Case study 1b: BESS control using DeePC and HDeePC with a nonlinear model.}}
    \label{fig:bessnlxp}
\end{figure}

\jdw{With this nonlinear system, we consider the quadratic cost \( J = \sum_{k=1}^{N} \big( y(k)^\top Q y(k) + u(k)^\top R u(k) \big) \), aiming to drive \(y_1\) and \(y_2\) to zero, with weighting matrices \(Q = 10^3 I_2\) and \(R = I_2\) (see GitHub for full implementation details). Subspace ID + MPC is clearly inferior (cost = $16.67 \times 10^4$), while HDeePC slightly outperforms DeePC (total cost of $2.186 \times 10^4$ compared to $2.196 \times 10^4$) as shown in Fig. \ref{fig:bessnlxp}. Although theoretical guarantees (i.e. Proposition \ref{thm:nlequiv}) do not hold in this case as the Fundamental Lemma is not satisfied for the unknown nonlinear part, HDeePC is still able to perform well with appropriate regularization while sys ID + MPC performs poorly\footnote{\jdw{Note that if $b$ is reduced to 0.1, thus making the system less nonlinear, sys ID + MPC is now able to reach a satisfactory solution only a little worse than DeePC or HDeePC}.} DeePC performs acceptably but is inferior to HDeePC as can also be seen from the oscillatory behaviour in Fig. \ref{fig:bessnlxp}.}

\subsection{CASE STUDY 2: TRIPLE-MASS SYSTEM}

Since the states in Case study 1 are uncoupled ($A$ in \eqref{eq:ex1ss} is block diagonal), we show the applicability of HDeePC to coupled triple-mass dynamics \cite{fiedler2021} using similar parameters ($T_{ini} = 4, N = 20, T = 150, \lambda_g = 1, \lambda_u = 10^6, \lambda_y = 10^6$, measurement noise uniformly distributed in $[-2 \times 10^5, 2 \times 10^5]$). Solving this problem over 40 timesteps gives the following results: DeePC time = 66.21s, cost = 9.41; HDeePC (equations for $x_3 \dots x_8$ and $y_3$ known) 49.51s, cost = 9.25; and MPC: 45.07s, cost = 9.25, where the advantage of using HDeePC over DeePC is clear. \jdw{We additionally illustrate Remark \ref{rmk:obs} by solving this problem using a certainty-equivalence approach with a partial Luenberger (for the known states only) observer, rather than full-state measurements, and the performance is not significantly degraded: 49.90s, cost = 9.35.}

\subsubsection*{\jdw{2a:} NUMBER OF KNOWN STATES}
For this experiment, we adjust $C = I_8, D = 0_{8\times 2}, Q = I_8$ and vary the number of state / output equations which are known. \jdw{The} results in Table \ref{tab:ex2nk} illustrate that HDeePC is a compromise between MPC and \jdw{DeePC}. 

\begin{table}[!ht]\centering
\caption{Case study \jdw{2a} results for varying number of known states}
\renewcommand{\arraystretch}{1.25}
\begin{tabular}{ c| c |c }
 \hline
 No. of known state equations & Total comp. time & Cost  \\
 \hline
 None (DeePC) & 82.18s & 86.96\\
 1 & 79.81s & 84.99\\
 2 & 77.82s & 85.69\\
 3 & 51.42s & 82.04\\
 4 & 48.95s & 82.04\\
 5 & 47.60s & 82.04\\
 6 & 48.17s & 82.04\\
 7 & 47.93s & 84.12\\
 8 (MPC \jdw{- full knowledge}) & 44.30s & 82.04\\
\end{tabular}
\label{tab:ex2nk}
\end{table}

\subsubsection*{\jdw{2b:} TIME-VARYING SYSTEM EXAMPLE}
In this case, the constant matrix $A$ from the previous example is changed into a function of time with entries of $A(k)$ being randomly (i.i.d. with 10\% standard deviation) varied in two scenarios: (Scenario 1) all entries are time-varying; (Scenario 2) only the last six state equations (these are the six state equations known to HDeePC) are time-varying. Table \ref{tab:ex2tv} shows the results. The fact that HDeePC incorporates knowledge (six state equations and one output equation) about the time-varying system allows it to perform considerably better than pure DeePC. \jdw{We also compare to subspace system identification (n4sid in MATLAB) + MPC which performs poorly as the subspace system identification does not handle the time-varying nature of the plant well.} 

\begin{table}[!ht]\centering
\caption{Case study \jdw{2b} results for time-varying systems}
\renewcommand{\arraystretch}{1.25}
\begin{tabular}{ c|c| c |c }
 \hline
 Scenario & Controller & Total computational time & Cost  \\
 \hline
 1 & DeePC & 70.56s & 17.94\\
 1 & HDeePC & 47.07s & 8.63\\
 1 & MPC\tablefootnote{\jdw{Full model knowledge.}} & 46.91s & 5.73\\
 1 & \jdw{ID+MPC} & \jdw{71.30s} & \jdw{59.25}\\
 \hline
 2 & DeePC & 80.64s & 41.00\\
 2 & HDeePC & 48.31s & 28.11\\
 2 & MPC\tablefootnote{\jdw{Full model knowledge.}} & 44.42s & 23.66 \\
 2 & \jdw{ID+MPC} & \jdw{71.51s} & \jdw{72.18}\\
\end{tabular}
\label{tab:ex2tv}
\end{table}

\subsection{\jdw{CASE STUDY 3: LARGER (POWER) SYSTEM EXAMPLE}}
\jdw{In order to demonstrate that HDeePC can be successfully applied / scaled to moderate-to-large systems, we consider a 33-bus AC/DC power system (Fig. \ref{fig:cs3}) with a mix of synchronous generators (SG), grid-forming (GFM) and grid-following converters, and DC-DC converters. The (reduced-order, LTI) system model was previously presented and explained in \cite{umang2024}. The model has 68 states, 10 inputs (reference DC voltages, real power, and DC current at various buses), and to simplify the application of HDeePC we include full state measurements. The HDeePC optimization problem size depends on the number of known states but ranges from 8295 to 19107 decision variables plus constraints. The full system matrices and controller parameters etc. may be found in the code on GitHub.}

\begin{figure}[!ht]
    \centering
    \includegraphics[width=0.488\textwidth]{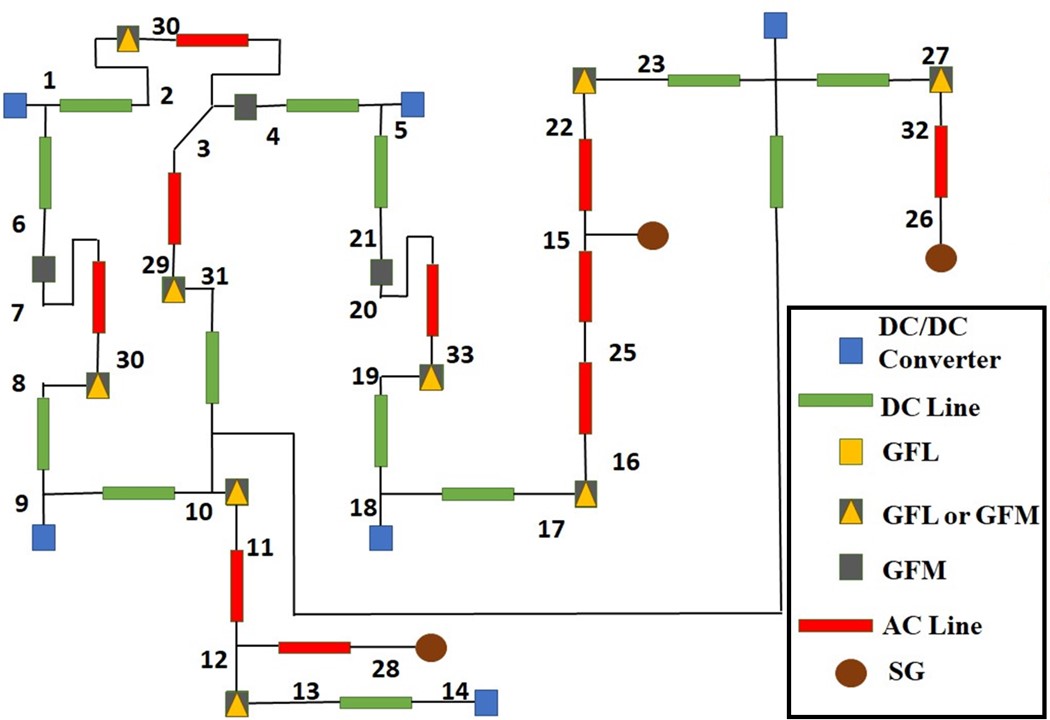}
    \caption{\jdw{Case study 3: AC/DC power system example.}}
    \label{fig:cs3}
\end{figure}

\jdw{We conduct an empirical study to evaluate the benefit of HDeePC in a larger-scale setting, varying both the number of known states and the degree of model mismatch. The simulation also includes measurement noise of magnitude (SD) $2 \times 10^{-5}$. Model mismatch is introduced in three ways: (i) parameter variation, where physical parameters such as RLC components, converter and generator time constants, and voltage levels are randomly perturbed according to specified standard deviations (SD), and the perturbed dynamical model is then used as the “true” system in simulation; (ii) multiplicative perturbation, where the state-space matrices themselves are perturbed elementwise; and (iii) low-rank additive perturbation, constructed as a rank-$r$ matrix and scaled so that its Frobenius norm equals a specified fraction of the nominal matrix norm, providing a clear measure of relative perturbation size. Results are reported in Table~\ref{tab:ex3a}, obtained using Mosek\footnote{\jdw{SDPT3 and SeDuMi also yield broadly similar results, but Mosek is clearly superior in terms of computational time and numerical stability.}} as the solver. Overall, the results show that HDeePC can significantly improve computational performance and cost relative to DeePC especially once a substantial amount of model knowledge is incorporated. At low levels of model knowledge (e.g., 75\% or more unknown states in this example), performance is less reliable, with some runs producing slightly higher costs than either pure MPC or DeePC. However, this variability may not be systematic: changing the weightings or form of the cost function often results in the scenarios with 51 or 60 unknown states outperforming DeePC; and in most cases HDeePC provides improvements over DeePC. The broader trend is that increasing model knowledge consistently reduces computational cost and often improves performance, with HDeePC converging toward MPC-like behavior in the limit. }

\begin{table*}[!ht]
\centering
\caption{\jdw{Case study 3a: larger system. Results: cost, computational time (average over 10 runs)}} 
\begin{threeparttable}
\renewcommand{\arraystretch}{1.25}
\begin{tabular}{c|c|c|c|c|c|c|c}
\hline
\jdw{Scenario} & \jdw{No mismatch} & \multicolumn{2}{c|}{\jdw{Parameter variation}} & \multicolumn{2}{c}{\jdw{Mult. Perturbation (SD)}} & \multicolumn{2}{|c}{\jdw{Additive Perturbation\tnote{a}}}\\
\hline
\jdw{\# unknown states} & \jdw{N/A} & \jdw{20\%} & \jdw{40\%} & \jdw{20\%} & \jdw{40\%} & \jdw{2\%} & \jdw{10\%}\\
\hline
\jdw{0 (MPC)} & \jdw{0.090, 1.27s} & \jdw{0.096, 2.43s} & \jdw{0.12, 1.37s} & \jdw{0.34, 2.16s} & \jdw{4.93, 1.97s} & \jdw{0.13, 1.21s} & \jdw{0.58, 1.51s}\\
\jdw{9 (13\%)} & \jdw{18.4, 4.75s} & \jdw{18.4, 8.79s} & \jdw{18.4, 5.04s} & \jdw{18.4, 8.94s} & \jdw{18.4, 5.24s} & \jdw{18.4, 7.04s} & \jdw{18.4, 5.80s}\\
\jdw{17 (25\%)} & \jdw{16.2, 6.17s} & \jdw{16.2, 11.1s} & \jdw{16.2, 6.55s} & \jdw{16.2, 7.04s} & \jdw{16.2, 6.81s} & \jdw{16.2, 6.72s} & \jdw{16.2, 7.36s}\\
\jdw{34 (50\%)} & \jdw{19.7, 9.87s} & \jdw{19.7, 17.8s} & \jdw{19.7, 10.8s} & \jdw{19.7, 10.9s} & \jdw{19.7, 10.2s} & \jdw{19.7, 11.9s} & \jdw{19.7, 14.9s}\\
\jdw{51 (75\%)} & \jdw{32.9, 18.8s} & \jdw{32.9, 34.5s} & \jdw{32.9, 19.6s} & \jdw{32.9, 21.1s} & \jdw{32.9, 19.5s} & \jdw{32.9, 22.6s} & \jdw{32.9, 29.4s}\\
\jdw{60 (88\%)} & \jdw{32.3, 59.2s} & \jdw{32.3, 91.1s} & \jdw{32.3, 54.9s} & \jdw{32.3, 62.2s} & \jdw{32.3, 58.8s} & \jdw{32.3, 71.6s} & \jdw{32.3, 73.6s}\\
\jdw{68 (DeePC)} & \jdw{32.6, 82.5s} & \jdw{32.6, 82.1s} & \jdw{32.6, 82.1s} & \jdw{32.6, 85.0s} & \jdw{32.6, 82.1s} & \jdw{32.6, 82.1s} & \jdw{32.6, 82.1s}\\
\hline
\end{tabular}
\begin{tablenotes}
    \item[a] \jdw{A random perturbation matrix is generated, with all entries set to zero except for four randomly selected positions on different rows, i.e. 4 of 68 state equations (similarly for output equations) are perturbed. The resulting sparse matrix is then scaled so that its Frobenius norm equals a specified percentage of the nominal matrix norm, and then added to the ``true" matrix.}
\end{tablenotes}
\end{threeparttable}
\label{tab:ex3a}
\end{table*}




\section{DISCUSSION} \label{discuss}

\subsection{\jdw{COMPARISON TO INDIRECT APPROACHES}}
\jdw{When applying predictive control to a plant with partially known dynamics, approaches are typically classified as indirect (system identification followed by MPC) or direct (e.g. DeePC, as there is no explicit identification step. Both approaches have advantages and disadvantages with respect to noise and bias (broadly speaking, indirect approaches are more resilient to noise but more sensitive to bias \cite{dorfler2023}). Generally speaking, in the LTI case there is no need to use direct data-driven control as LTI system identification is well established and reliable \cite{morato2024} as is consistent with our results in Case study 1a. HDeePC is a direct method that aims to improve upon DeePC (when possible) by incorporating the part of the model that is known, retaining the advantages of a direct approach. The motivation for HDeePC is particularly strong in cases where partial model knowledge captures key nonlinearities or time-varying dynamics as we show in case studies 1b, 1c and 2. These scenarios are practically relevant; the development of HDeePC was driven by challenges in applying DeePC to power system and energy storage applications. }

\subsection{\jdw{NUMBER OF KNOWN/UNKNOWN STATES}}
\jdw{In case studies 2 and especially 3, we investigated the effect of varying the order of the known subsystem (i.e. the number of known states). It should be noted that the results are somewhat dependent on regularization choices and choices of parameters such as $T, T_{ini}, N$. For the results in Tables \ref{tab:ex2nk} and \ref{tab:ex3a}, these have been kept consistent to allow a fair comparison, however we believe it is possible to tune regularization terms etc. to improve any particular result. The results in the larger example (Case study 3) clearly show the benefit of using HDeePC if enough of a partial model is known, even if this subject to reasonable levels of uncertainty.}

\subsection{\jdw{MODEL UNCERTAINTY OR MISMATCH}}
\jdw{If the model is uncertain there is a trade-off between enlarging the known subsystem and preserving the flexibility of the data-driven part. Making more of the model known reduces the dimension of the Hankel matrices, improves noise resilience, and moves the formulation closer to the efficiency and robustness of MPC, provided the model is accurate and states are measurable. On the other hand, leaving more dynamics in the data-driven block increases flexibility, allowing the controller to adapt to unmodeled effects, nonlinearities, or parameter drift without risking bias from an incorrect model. The practical balance depends on how much of the system can be modeled with sufficient confidence versus how much uncertainty or complexity is better handled by the data-based part representation. When significant model uncertainty is present, standard robust MPC techniques can be applied to the model-based component of HDeePC. Moreover, distributionally robust optimization, as proposed for DeePC in \cite{coulson2022} and for a similar ``hybrid" formulation in \cite{li2025mdrdeepcmodelinspireddistributionallyrobust}, can also be incorporated. In Case study 3, practical levels of model mismatch are generally well tolerated, except in scenarios such as significant non-sparse additive perturbations, where the structure of the system dynamics is substantially altered\footnote{\jdw{In our experiments, dense additive perturbations of sufficient magnitude (around 1\% scaled by the Frobenius norm) may result in poor control or even instability.}}. Nevertheless, it is often reasonable to assume in practice that the structural form of the system dynamics is correctly specified, with uncertainties arising predominantly from parameter variations rather than from (significant) structural discrepancies. }

\section{CONCLUSION}\label{conc}
This paper has proposed HDeePC, a novel form of predictive control which uses both data and model information. Conditions have been derived under which we are able to prove feasible set equivalence and equivalent closed-loop behavior for the proposed HDeePC. Case studies (including a BESS in a microgrid\jdw{, and a larger hybrid AC/DC power system}) have been given to demonstrate the advantages in some cases of HDeePC \jdw{compared to DeePC and subspace ID + MPC}. Future work involves investigating \jdw{regularization selection in the context of HDeePC, considering} other formulations to incorporate partial model knowledge into DeePC/HDeePC, \jdw{analytically addressing the effect of measurement noise and plant-model mismatch on HDeePC, analyzing the choice of known/unknown states from a theoretical perspective, and relaxing the fairly restrictive theoretical conditions with respect to applying HDeePC to non-linear systems.}

\section*{ACKNOWLEDGMENTS}
\jdw{The author thanks Cameron Mantell for useful discussions.}

\bibliographystyle{IEEEtran}
\bibliography{refs}   

\begin{thebibliography}{10}
\providecommand{\url}[1]{#1}
\csname url@samestyle\endcsname
\providecommand{\newblock}{\relax}
\providecommand{\bibinfo}[2]{#2}
\providecommand{\BIBentrySTDinterwordspacing}{\spaceskip=0pt\relax}
\providecommand{\BIBentryALTinterwordstretchfactor}{4}
\providecommand{\BIBentryALTinterwordspacing}{\spaceskip=\fontdimen2\font plus
\BIBentryALTinterwordstretchfactor\fontdimen3\font minus \fontdimen4\font\relax}
\providecommand{\BIBforeignlanguage}[2]{{%
\expandafter\ifx\csname l@#1\endcsname\relax
\typeout{** WARNING: IEEEtran.bst: No hyphenation pattern has been}%
\typeout{** loaded for the language `#1'. Using the pattern for}%
\typeout{** the default language instead.}%
\else
\language=\csname l@#1\endcsname
\fi
#2}}
\providecommand{\BIBdecl}{\relax}
\BIBdecl

\bibitem{hjalmarsson2005}
N.~Hjalmarsson, ``From experiment design to closed-loop control,'' \emph{Automatica}, vol.~41, no.~3, pp. 393--438, 2005.

\bibitem{dorfler2023}
F.~Dörfler, J.~Coulson, and I.~Markovsky, ``\jdw{Bridging Direct and Indirect Data-Driven Control Formulations via Regularizations and Relaxations},'' \emph{IEEE Transactions on Automatic Control}, vol.~68, no.~2, pp. 883--897, 2023.

\bibitem{coulson2019}
J.~Coulson, J.~Lygeros, and F.~Dörfler, ``Data-enabled predictive control: In the shallows of the deepc,'' in \emph{European Control Conference (ECC)}, 2019, pp. 307--312.

\bibitem{coulson2022}
------, ``Distributionally robust chance constrained data-enabled predictive control,'' \emph{IEEE Transactions on Automatic Control}, vol.~67, no.~7, pp. 3289--3304, July 2022.

\bibitem{markowsky2006}
I.~Markovsky, J.~C. Willems, S.~V. Huffel, and B.~D. Moor, \emph{Exact and Approximate Modeling of Linear Systems: A Behavioral Approach}.\hskip 1em plus 0.5em minus 0.4em\relax SIAM, 2006.

\bibitem{depersis2020}
C.~D. Persis and P.~Tesi, ``Formulas for data-driven control: Stabilization, optimality, and robustness,'' \emph{IEEE Transactions on Automatic Control}, vol.~65, no.~3, pp. 909--924, March 2020.

\bibitem{huang2019}
L.~Huang, J.~Coulson, J.~Lygeros, and F.~Dörfler, ``Data-enabled predictive control for grid-connected power converters,'' in \emph{IEEE Conference on Decision and Control (CDC)}, 2019, pp. 8130--8135.

\bibitem{huang2022}
------, ``Decentralized data-enabled predictive control for power system oscillation damping,'' \emph{IEEE Transactions on Control Systems Technology}, vol.~30, no.~3, pp. 1065--1077, 2022.

\bibitem{rimoldi2024}
A.~Rimoldi, C.~Cenedese, A.~Padoan, F.~Dörfler, and J.~Lygeros, ``Urban traffic congestion control: A deepc change,'' in \emph{2024 European Control Conference (ECC)}, 2024, pp. 2909--2914.

\bibitem{elokda2021}
E.~Elokda, J.~Coulson, P.~Beuchat, J.~Lygeros, and F.~Dörfler, ``Data-enabled predictive control for quadcopters,'' \emph{International Journal of Robust and Nonlinear Control}, July 2021.

\bibitem{carlet2020}
P.~G. Carlet, A.~Favato, S.~Bolognani, and F.~Dörfler, ``Data-driven continuous-set predictive current control for synchronous motor drives,'' \emph{IEEE Transactions on Power Electronics}, vol.~37, no.~6, pp. 6637--6646, 2022.

\bibitem{berberich2021}
J.~Berberich, J.~Köhler, M.~A. Müller, and F.~Allgöwer, ``Data-driven model predictive control with stability and robustness guarantees,'' \emph{IEEE Transactions on Automatic Control}, vol.~66, no.~4, pp. 1702--1717, 2021.

\bibitem{berberich2025}
J.~Berberich and F.~Allgöwer, ``An overview of systems-theoretic guarantees in data-driven model predictive control,'' \emph{Annual Review of Control, Robotics, and Autonomous Systems}, vol.~8, 2025.

\bibitem{li2025review}
\BIBentryALTinterwordspacing
X.~Li, M.~Yan, X.~Zhang, M.~Han, A.~W.-K. Law, and X.~Yin, ``\jdw{Efficient data-driven predictive control of nonlinear systems: A review and perspectives},'' \emph{Digital Chemical Engineering}, vol.~14, p. 100219, 2025. [Online]. Available: \url{https://www.sciencedirect.com/science/article/pii/S2772508125000031}
\BIBentrySTDinterwordspacing

\bibitem{zhang2023}
K.~Zhang, Y.~Zheng, C.~Shang, and Z.~Li, ``Dimension reduction for efficient data-enabled predictive control,'' \emph{IEEE Control Systems Letters}, vol.~7, pp. 3277--3282, 2023.

\bibitem{vahidimoghaddam2024onlinereducedorderdataenabledpredictive}
\BIBentryALTinterwordspacing
A.~Vahidi-Moghaddam, K.~Zhang, X.~Yin, V.~Srivastava, and Z.~Li, ``\jdw{Online Reduced-Order Data-Enabled Predictive Control},'' 2024. [Online]. Available: \url{https://arxiv.org/abs/2407.16066}
\BIBentrySTDinterwordspacing

\bibitem{zhou2025}
Y.~Zhou, Y.~Lu, Z.~Li, J.~Yan, and Y.~Mo, ``Learning-based efficient approximation of data-enabled predictive control,'' in \emph{2024 IEEE 63rd Conference on Decision and Control (CDC)}, 2024, pp. 322--327.

\bibitem{huang2023}
L.~Huang, J.~Zhen, J.~Lygeros, and F.~Dörfler, ``\jdw{Robust Data-Enabled Predictive Control: Tractable Formulations and Performance Guarantees},'' \emph{IEEE Transactions on Automatic Control}, vol.~68, no.~5, pp. 3163--3170, 2023.

\bibitem{lian2021}
\BIBentryALTinterwordspacing
Y.~Lian and C.~N. Jones, ``\jdw{Nonlinear Data-Enabled Prediction and Control},'' in \emph{Proceedings of the 3rd Conference on Learning for Dynamics and Control}, ser. Proceedings of Machine Learning Research, A.~Jadbabaie, J.~Lygeros, G.~J. Pappas, P.~A.~Parrilo, B.~Recht, C.~J. Tomlin, and M.~N. Zeilinger, Eds., vol. 144.\hskip 1em plus 0.5em minus 0.4em\relax PMLR, 07 -- 08 June 2021, pp. 523--534. [Online]. Available: \url{https://proceedings.mlr.press/v144/lian21a.html}
\BIBentrySTDinterwordspacing

\bibitem{berberich2022}
J.~Berberich, J.~Köhler, M.~A. Müller, and F.~Allgöwer, ``\jdw{Linear Tracking MPC for Nonlinear Systems—Part II: The Data-Driven Case},'' \emph{IEEE Transactions on Automatic Control}, vol.~67, no.~9, pp. 4406--4421, 2022.

\bibitem{lazar2024}
M.~Lazar, ``\jdw{Basis-Functions Nonlinear Data-Enabled Predictive Control: Consistent and Computationally Efficient Formulations},'' in \emph{2024 European Control Conference (ECC)}, 2024, pp. 888--893.

\bibitem{berberich2024}
\BIBentryALTinterwordspacing
J.~Berberich and F.~Allgöwer, ``\jdw{An Overview of Systems-Theoretic Guarantees in Data-Driven Model Predictive Control},'' \emph{Annual Review of Control, Robotics, and Autonomous Systems}, vol.~8, no. Volume 8, 2025, pp. 77--100, 2025. [Online]. Available: \url{https://www.annualreviews.org/content/journals/10.1146/annurev-control-030323-024328}
\BIBentrySTDinterwordspacing

\bibitem{huang2024}
L.~Huang, J.~Lygeros, and F.~Dörfler, ``\jdw{Robust and Kernelized Data-Enabled Predictive Control for Nonlinear Systems},'' \emph{IEEE Transactions on Control Systems Technology}, vol.~32, no.~2, pp. 611--624, 2024.

\bibitem{naf2025}
\BIBentryALTinterwordspacing
J.~Näf, K.~Moffat, J.~Eising, and F.~Dörfler, ``\jdw{Choose Wisely: Data-driven Predictive Control for Nonlinear Systems Using Online Data Selection},'' 2025. [Online]. Available: \url{https://arxiv.org/abs/2503.18845}
\BIBentrySTDinterwordspacing

\bibitem{morato2024}
\BIBentryALTinterwordspacing
M.~M. Morato and M.~S. Felix, ``Data science and model predictive control: A survey of recent advances on data-driven mpc algorithms,'' \emph{Journal of Process Control}, vol. 144, p. 103327, 2024. [Online]. Available: \url{https://www.sciencedirect.com/science/article/pii/S0959152424001677}
\BIBentrySTDinterwordspacing

\bibitem{baros2022}
S.~Baros, C.-Y. Chang, G.~E. Colón-Reyes, and A.~Bernstein, ``Online data-enabled predictive control,'' \emph{Automatica}, vol. 138, p. 109926, 2022.

\bibitem{berberich2020}
J.~Berberich, A.~Koch, C.~W. Scherer, and F.~Allgöwer, ``\jdw{Robust data-driven state-feedback design},'' in \emph{2020 American Control Conference (ACC)}, 2020, pp. 1532--1538.

\bibitem{djeumou2022}
F.~Djeumou and U.~Topcu, ``Learning to reach, swim, walk and fly in one trial: Data-driven control with scarce data and side information,'' \emph{Proceedings of Machine Learning Research}, vol. 168, pp. 453--466, 2022.

\bibitem{watson2025hybriddataenabledpredictivecontrol}
\BIBentryALTinterwordspacing
J.~D. Watson, ``Hybrid data-enabled predictive control: Incorporating model knowledge into the deepc,'' 2025. [Online]. Available: \url{https://arxiv.org/abs/2502.12467}
\BIBentrySTDinterwordspacing

\bibitem{zieglmeier2025semidatadrivenmodelpredictivecontrol}
\BIBentryALTinterwordspacing
S.~Zieglmeier, M.~H. de~Badyn, N.~D. Warakagoda, T.~R. Krogstad, and P.~Engelstad, ``\jdw{Semi-Data-Driven Model Predictive Control: A Physics-Informed Data-Driven Control Approach},'' 2025. [Online]. Available: \url{https://arxiv.org/abs/2504.00746}
\BIBentrySTDinterwordspacing

\bibitem{li2025mdrdeepcmodelinspireddistributionallyrobust}
\BIBentryALTinterwordspacing
S.~Li, J.~Li, C.~Martin, S.~Bakshi, and D.~Chen, ``\jdw{MDR-DeePC: Model-Inspired Distributionally Robust Data-Enabled Predictive Control},'' 2025. [Online]. Available: \url{https://arxiv.org/abs/2506.19744}
\BIBentrySTDinterwordspacing

\bibitem{fiedler2021}
F.~Fiedler and S.~Lucia, ``On the relationship between data-enabled predictive control and subspace predictive control,'' in \emph{2021 European Control Conference (ECC)}, 2021, pp. 222--229.

\bibitem{kuntz2024}
S.~J. Kuntz and J.~B. Rawlings, ``\jdw{Beyond inherent robustness: strong stability of {MPC} despite plant-model mismatch},'' 2024.

\bibitem{Toh2012}
K.-C. Toh, M.~J. Todd, and R.~H. T{\"u}t{\"u}nc{\"u}, \emph{On the Implementation and Usage of SDPT3 -- A Matlab Software Package for Semidefinite-Quadratic-Linear Programming, Version 4.0}.\hskip 1em plus 0.5em minus 0.4em\relax New York, NY: Springer US, 2012, pp. 715--754.

\bibitem{MATTSSON2023625}
P.~Mattsson and T.~B. Schön, ``On the regularization in deepc*,'' \emph{IFAC-PapersOnLine}, vol.~56, no.~2, pp. 625--631, 2023, 22nd IFAC World Congress.

\bibitem{umang2024}
U.~A. Wasekar and J.~D. Watson, ``\jdw{Monte-Carlo analysis of interlinking converter modelling and control in hybrid AC/DC networks},'' in \emph{2024 18th International Conference on Probabilistic Methods Applied to Power Systems (PMAPS)}, 2024, pp. 1--6.

\end{thebibliography}

\section*{APPENDIX: Proof of Theorem \ref{thm:fse}}

\begin{proof}
We start by considering the data-based and model-based constraints in turn. Following the steps of \cite[Theorem 5.1]{coulson2019} in conjunction with the stated assumption that $u^d$ is persistently exciting shows that the data-based constraint \eqref{eq:hdpcdata} gives the feasible set of the set of pairs $(u, y_u)$ satisfying:
\begin{equation}\label{eq:fsc}
\begin{aligned}
    y_u &= \mathscr{O}_N(A, [C_u \quad C_f])x_{ini} \\&+ \mathscr{T}_N([A, B, [C_u \quad C_f], D_u])u, 
\end{aligned}
\end{equation}
where $y_\kappa$ is unconstrained by \eqref{eq:fsc}. For illustration, this may be written as:
\begin{equation}\label{eq:fsc2}
\begin{aligned}
    y_u(1) &= [C_u \quad C_f] x_{ini} + D_uu(1) \\
    y_u(2) &= [C_u \quad C_f] A x_{ini} + [C_u \quad C_f] B u(1) + D_uu(2) \\
    y_u(3) &= [C_u \quad C_f] A^2 x_{ini} +[C_u \quad C_f] A B u(1)\\ &+ [C_u \quad C_f] B u(2) + D_uu(3) \\
    \vdots
\end{aligned}
\end{equation}
Similarly, the feasible set of the model-based constraints \eqref{eq:hdpcss} can be written as the set of pairs $(u, y)$ that satisfy:
\begin{equation}\label{eq:fsk0}
\begin{aligned}
    y_\kappa (1) &= C_\kappa x_{\kappa,ini} + C_yy_u(1) + D_\kappa u(1) \\
    y_\kappa(2) &= C_\kappa A_\kappa x_{\kappa,ini}  + C_\kappa A_yy_u(1) + C_yy_u(2)\\
           &+ C_\kappa B_\kappa u(1) + D_\kappa u(2) \\
    y_\kappa(3) &= C_\kappa A_\kappa^2 x_{k,ini} + C_\kappa A_\kappa A_yy_u(1) + C_\kappa A_yy_u(2) + C_yy_u(3) \\
           &+ C_\kappa A_\kappa B u(1) + C_\kappa B u(2) + D_\kappa u(3) \\
    \vdots
\end{aligned}
\end{equation}
From $A_cx_u = A_yy_u = A_yC_ux_u$ and $C_cx_u = C_yy_u = C_yC_ux_u$ (Assumption \ref{as:transform}), it follows that $A_yC_fx_\kappa$, $A_yD_uu$, $C_yC_fx_\kappa$, and $C_yD_uu$ are all zero. Using this and substituting \eqref{eq:fsc2} into \eqref{eq:fsk0}:
\begin{equation}
\begin{aligned}
    y_\kappa(1) &= [C_c \quad C_\kappa] x_{ini} + D_\kappa u(1) \\
    y_\kappa(2) &= [C_c \quad C_\kappa] A x_{ini} + [C_c \quad C_\kappa] B u(1) + D_\kappa u(2) \\
    y_\kappa(3) &= [C_c \quad C_\kappa] A^2 x_{ini} +[C_c \quad C_\kappa] A B u(1)\\ &+ [C_c \quad C_\kappa] B u(2) + D_\kappa u(3) \\
    \vdots
\end{aligned}
\end{equation}
This can be written as:
\begin{equation}\label{eq:fsk}
\begin{aligned}
    y_\kappa &= \mathscr{O}_N(A, \begin{bmatrix}
        C_c & C_\kappa
    \end{bmatrix})x_{ini} \\&+ \mathscr{T}_N(A, B, \begin{bmatrix}
        C_c & C_\kappa
    \end{bmatrix}, D_\kappa)u
\end{aligned}
\end{equation}
Putting \eqref{eq:fsc} and \eqref{eq:fsk} together gives:
\begin{equation}\label{eq:fscombined}
\begin{aligned}
    \begin{bmatrix}y_u\\y_\kappa\end{bmatrix} &= \begin{bmatrix}\mathscr{O}_N(A, \begin{bmatrix}
        C_u & C_f
    \end{bmatrix}) \\ \mathscr{O}_N(A, \begin{bmatrix}
        C_c & C_\kappa
    \end{bmatrix})\end{bmatrix}x_{ini} \\
    &+ \begin{bmatrix} \mathscr{T}_N(A, B, [C_u \quad C_f], D_u)
        \\\mathscr{T}_N(A, B, \begin{bmatrix}
        C_c & C_\kappa
    \end{bmatrix}, D_\kappa)\end{bmatrix}u
\end{aligned}
\end{equation}
which after rearranging, gives:
\begin{equation*}
    y = \mathscr{O}_N(A,C)x_{ini} + \mathscr{T}_N(A,B,C,D)u,
\end{equation*}
with $A = \begin{bmatrix}
            A_u & A_f \\
            A_c & A_\kappa
        \end{bmatrix}, B = \begin{bmatrix}
            B_u \\
            B_\kappa
        \end{bmatrix}, C = \begin{bmatrix}
            C_u & C_f \\
            C_c & C_\kappa
        \end{bmatrix}, D = \begin{bmatrix}
            D_u \\
            D_\kappa
        \end{bmatrix}$.
To find the feasible set, we add the input and output constraints $u \in \mathcal{U}^N, y \in \mathcal{Y}^N$, yielding the feasible set as the set of pairs $(u, y) \in \{u \in \mathcal{U}^N, y \in \mathcal{Y}^N\}$ that satisfy:
\begin{equation*}
    y = \mathscr{O}_N(A,C)x_{ini} + \mathscr{T}_N(A,B,C,D)u,
\end{equation*}
 This is equivalent to the feasible set of both DeePC \eqref{eq:deepc} and MPC \eqref{eq:mpc} \cite[Theorem 5.1]{coulson2019}, completing the proof. 
\end{proof}

\end{document}